\documentclass[11pt, twoside]{amsart}
\usepackage[utf8]{inputenc}
\usepackage[toc,page]{appendix}
\usepackage{makecell}
\usepackage{stackrel}
\usepackage{tikz-cd}
\usepackage{amssymb}

\usepackage{cellspace}
\usepackage{array}
\usepackage{multirow}
\usepackage{adjustbox}
\newcolumntype{C}{>{$} Sc <{$} }
\allowdisplaybreaks

\usepackage[
backend=biber,
style=alphabetic,
sorting=nyt,
maxbibnames=99
]{biblatex}

\allowdisplaybreaks

\usepackage{imakeidx}
\makeindex

\usepackage{amsmath,amsfonts,amsthm,amssymb,mathrsfs}
\usepackage{hyperref}
\usepackage{graphicx}
\usepackage{multirow}
\usepackage{tikz-cd}
\usepackage{enumerate}
\usepackage{amscd}
\usepackage{enumitem}

\usepackage[margin=.75in]{geometry}

\usepackage[hyphenbreaks]{breakurl}

\definecolor{gold}{rgb}{0.88,0.68,0.13}
\definecolor{dgreen}{rgb}{0.0,0.40,0.13}
\definecolor{dblue}{rgb}{0.20,0.20,0.80}

\newcommand{\te}[1]{\operatorname{#1}}

\newcommand{\nn}{\mathbb{N}}
\newcommand{\tn}{\mathbb{T}}
\newcommand{\zn}{\mathbb{Z}}

\newcommand{\Prim}{\operatorname{Prim}}
\newcommand{\Rep}{\operatorname{Rep}}
\newcommand{\Repn}{\te{Rep}_n}

\newcommand{\Irrn}{\te{Irr}_n}
\newcommand{\sse}{\subseteq}
\newcommand{\mc}{\mathcal}
\newcommand{\mt}{\mathtt}
\newcommand{\ds}{\displaystyle}
\newcommand{\ov}{\overline}

\newcommand{\wh}{\widehat}

\newcommand{\hn}{\mc{H}_n}

\renewcommand{\l}{\left}
\renewcommand{\r}{\right}

\newcommand{\usf}{\textsf{U}}
\newcommand{\bsf}{\textsf{B}}

\newcommand{\dimnuc }{\dim_{\operatorname{nuc}}}

\newtheorem{thm}{Theorem}[section]
\newtheorem{prop}[thm]{Proposition}

\newtheorem{cor}[thm]{Corollary}
\newtheorem{lemma}[thm]{Lemma}
\newtheorem{claim}{Claim}

\theoremstyle{definition}

\newtheorem{defn}[thm]{Definition}

\newtheorem{ex}[thm]{Example}

\newcounter{nitem}

\newcounter{alphaitem}

\newtheorem{rmk}[thm]{Remark}

\renewcommand{\thethm}{\Alph{thm}}  % Theorem A, B, C...

\newcommand{\on}{\operatorname}

\newcommand{\B}{\mathbb}

\newcommand{\N}{\B{N}}
\newcommand{\Z}{\B{Z}}

\newcommand{\C}{\B{C}}

\usepackage{bbm}

\newcommand{\T}{\B{T}}

\newcommand{\str}{\texorpdfstring}

\usepackage{ stmaryrd }

\newcommand{\ZZ}{\mathcal{Z}}

\newcommand{\TT}{\mathbb{T}}

\newcommand\nnfootnote[1]{%
  \begin{NoHyper}
  \renewcommand\thefootnote{}\footnote{#1}%
  \addtocounter{footnote}{-1}%
  \end{NoHyper}
}

\begin{document}
\title{Nuclear Dimension and Rigidity Results for Virtually Abelian Groups}

\author{Frankie Chan}
\address{FC: Department of Mathematics, University of Chicago, Chicago IL 60637, USA}
\email[]{frankiechan@uchicago.edu}

\author{S. Joseph Lippert}
\address{JL: Department of Mathematics and Statistics, Sam Houston State University, Huntsville TX 77341, USA}
\email[]{sjl054@shsu.edu}

\author{Iason Moutzouris}
\address{IM: Department of Mathematics and Statistics, Sam Houston State University, Huntsville TX 77341, USA}
\email[]{ixm089@shsu.edu \hspace{10mm} %Website:https://sites.google.com/view/iasonmoutzouris/
}

\author{Ellen Weld}
\address{EW: Department of Mathematics and Statistics, Sam Houston State University, Huntsville TX 77341, USA}
\email[]{elw028@shsu.edu}
\date{\today}
\nnfootnote{2020 Mathematics Subject Classification.22D10,22D25,46L05.}

\begin{abstract}
Let $G$ be a finitely generated virtually abelian group. We show that the Hirsch length, $h(G)$, is equal to the nuclear dimension of its group $C^*$-algebra, $\dimnuc (C^*(G))$. We then specialize our attention to a generalization of crystallographic groups dubbed \textit{crystal-like}. We demonstrate that in this scenario a \textit{point group} is well defined and the order of this point group is preserved by $C^*$-isomorphism. We close by using these tools to demonstrate that crystallographic (as a group property) is preserved by $C^*$-isomorphism. These three tools combine to prove that 2D crystallographic groups are $C^*$-superrigid.
\end{abstract}

\maketitle

%%%%%%%%%%%%%%%%%%%%%%%%%%%%%%%%%%%%%%%%%%%%
%%%%%%%%%%%%%%%%%%%%%%%%%%%%%%%%%%%%%%%%%%%%

\section{Introduction}

A question of particular note in the realm of group $C^*$-algebras is that of group invariants recoverable within the algebra. Put directly, if we fix a discrete group $G$ and take any group $H$ such that $C^*(G)\cong C^*(H)$, what can be said about the relationship between $G$ and $H$? In the literature, these questions are referred to as (super)-rigidity questions. In particular, $G$ can be fully recovered (i.e., $G\cong H$) if, for example, $G$ is torsion-free, finitely generated, 2-step nilpotent \cite{tf_2_step_c*_superrigid}, free nilpotent \cite{free_nilp_c*_superrigid}, or belongs to a certain class of Bieberbach groups \cite{Knubyetal}. Moreover, it is known that $G$ and $H$ have the same first Betti numbers \cite{free_nilp_c*_superrigid}. \par

In this article, we narrow our focus to group $C^*$-algebras constructed from finitely generated virtually abelian groups. In this setting, there is a natural concept of dimension for the group called the Hirsch length. This dimension is equal to the rank of a normal abelian subgroup of finite index. For the definition of the Hirsch length in a larger class of groups, we refer the reader to \cite{hillman_finite_Hirsch_length_old}. Our main result draws a direct connection between the Hirsch length of a group and the nuclear dimension of its $C^*$-algebra. 

\begin{thm} (Theorem \ref{main_result_2})
Let $G$ be a discrete, finitely generated, virtually abelian group. Then $\dimnuc C^*(G)=h(G)$.
\end{thm}

Nuclear dimension is of additional note outside of the strict question of rigidity. Nuclear dimension plays an  important role on the classification of simple $C^*$-algebras \cite{GLN1,GLN2,tikuisis2017quasidiagonality,elliot2015quasidiagonality}. Indeed, in the case of simple, separable and nuclear $C^*$-algebras, the nuclear dimension can be $0$ (if the $C^*$-algebra is an AF-algebra), $1$ (if it absorbs tensorially the Jiang-Su algebra $\ZZ$) or $+\infty$ (otherwise) \cite{dim_nuc_simple,dim_nuc_simple_st_proj}. \par

In addition, finding the precise value of the nuclear dimension of a (non-simple) $C^*$-algebra has been a very challenging question. In the context of group $C^*$-algebras, Eckhardt and Wu proved \cite{EW24} that every virtually polycyclic group has finite nuclear dimension, generalizing previous results from \cite{EGM19,EM18}. In fact, they found upper bounds that depend only on the Hirsch length of the group. On the other hand, Giol and Kerr proved that $C^*(\Z \wr \Z)$ has infinite nuclear dimension \cite{GK10}. The group $\Z\wr \Z$ has infinite Hirsch length, so a more general connection between nuclear dimension and Hirsch length seems to exist. \par

Returning to the context of rigidity, a corollary to our main theorem is that for finitely generated and virtually abelian groups $C^*(G)\cong C^*(H)$ implies $h(G)=h(H)$. Seeking more results such as this (with special attention towards crystallographic groups), we define a notion of crystal-like for which index (of a particular abelian subgroup) is shown to be invariant. This result then inspires our second main theorem which generalizes work of Curda, Knuby, Raum, Thiel, and White. 

\begin{thm} (Theorem \ref{crystallographic_is_preserved})
Let $G$ be a crystallographic group and $H$ a discrete group such that $C^*(G)\cong C^*(H)$. Then $H$ is crystallographic (of the same dimension and point group order as $G$).
\end{thm}

The crystallographic groups of Hirsch length two \footnote{The Hirsch length in a crystallographic group is also called dimension.} are called \emph{wallpaper groups}. There are 17 wallpaper groups. We close by using Theorem \ref{crystallographic_is_preserved} (and well established $C^*$-invariants) to demonstrate $C^*$-superrigidity of all wallpaper groups. Of particular note are the $15$ wallpaper groups with torsion, these are among the first known examples of infinite, amenable, groups with torsion demonstrating $C^*$-superrigidity.

\section*{Acknowledgements}
We would like to thank Jakub Curda for pointing out Proposition \ref{computing the center of certain va groups} to us. We would also like to thank the referee for their many helpful comments and careful reading of our article.

\section{Preliminaries}

\renewcommand{\thethm}{\thesection.\arabic{thm}}

\subsection{Irreducible representations and subhomogeneous \str{$C^*$}{C*}-algebras.}

In this subsection, we give background information regarding the spectrum of $C^*$-algebras and its topology. For more details, we recommend the classic text {\em $C^*$-algebras} by Dixmier (\cite{Dix77}).

Let $A$ be a $C^*$-algebra. A two-sided ideal of $A$ is said to be {\em primitive} if it is the kernel of a non-zero irreducible representation of $A$ on some Hilbert space. The set of all primitive ideals of $A$ is denoted by $\Prim(A)$ and we endow it with the {\em Jacobson topology}. When given the Jacobson topology, we call $\Prim(A)$ the {\em primitive spectrum of $A$}. If $J\in\Prim(A)$ is the kernel of a dimension $k$ irreducible representation, then we say $\dim J=k$. In particular, we let
\[\Prim_k(A)=\{J\in\Prim (A)\,\colon\, \dim J=k\}.\]
Two irreducible representations $\pi:A\to \bsf(\mc{H})$ and $\pi':A\rightarrow \bsf(\mc{H}')$ are \emph{equivalent} if there exists a unitary operator $U:\mc{H}\to \mc{H}'$ such that $U\pi(a)=\pi'(a)U$ for all $a\in A$. In this case we write $\pi\simeq \pi'$. The {\em spectrum of $A$}, denoted by $\wh{A}$, is the set of non-zero irreducible representations under equivalence ($\pi'\in[[\pi]]\in\wh{A}\iff\pi\simeq\pi'$). This set is endowed with the inverse image of the Jacobson topology under the canonical map $\wh{A}\ni[[\pi]]\mapsto\ker\pi\in\Prim(A)$. \par
We fix the {\em standard Hilbert space of dimension $n$},\label{rmk:standard_hilbert} denoted by $\mc{H}_n$ (with standard basis $\{e_1,\dots,e_n\}$), for each $n\in\zn_{>0}$. We let $\Repn(A)$ be the set of representations of $A$ on $\hn$ and set $\Irrn(A)\sse \Repn(A)$ to be those irreducible representations of dimension $n$. We topologize $\Repn(A)$ (and thus $\Irrn(A)$) by weak pointwise convergence over $A$; that is, $\pi_k\rightarrow\pi$ for $\pi_k,\pi\in\Repn(A)$ means
\[\langle \pi_k(a)\xi,\eta\rangle_{\hn}\rightarrow \langle \pi(a)\xi,\eta\rangle_{\hn}\quad\text{for any }a\in A,\xi,\eta\in\hn.\]
\cite[Prop 3.7.1, 3.7.4]{Dix77} shows that $\Repn(A)$ and $\Irrn(A)$ are separable and completely metrizable.

A $C^*$-algebra $A$ is called \emph{subhomogeneous} if it embeds on a $C^*$-algebra of the form $C(X,M_n)$ for some compact, Hausdorff space $X$ and some $n\in \N$. Equivalently, $A$ is subhomogeneous if there exists $M>0$ such that every irreducible representation of $A$ has dimension $\leq M$. If $A$ is subhomogeneous, then $\widehat{A}\cong \te{Prim}\,(A)$ via the above canonical map (see \cite[3.1.6 (p.71)]{Dix77} and \cite[Thm IV.15.7 (p.339)]{Bla10}). \par

\begin{rmk}\label{rmk:prim_hausdorff}
When $A$ is subhomogeneous, $\te{Prim}_n(A)$ is Hausdorff for each $n\in\zn_{>0}$ \cite[p.83]{Dix77} but $\te{Prim}(A)$ is not Hausdorff in general.
\end{rmk}

\subsection{Pontryagin Dual.}
The {\em Pontryagin dual} of a discrete abelian group $G$ is the set $\widehat{G}:=\on{Hom}(G,\TT)$ endowed with pointwise multiplication and with the topology of pointwise convergence. With this topology, $\widehat{G}$ is compact and Hausdorff. As topological groups, we have $\widehat{\Z^n}=\TT^n$ and $\widehat{\Z_m}=\Z_m$, interpreting the latter as the group of the $m^{\te{th}}$ roots of unity.
\par
Since $\widehat{G\times H}=\widehat{G}\times \widehat{H}$, it follows that for a discrete finitely generated abelian group $A\cong\Z^r\times T$ (where $T$ is a torsion subgroup), we have that $\wh{A}\cong\T^r\times T$. Defining $\rho:\wh{A}\rightarrow \tn^r$ by $\rho(\chi):=\chi|_{\zn^r}$, a sequence $\{\chi_n\}\sse \wh{A}$ converges to $\chi\in \wh{A}$ if and only if (1) $\rho(\chi_n)\rightarrow\rho(\chi)\in\tn^r$ and (2) eventually ${\chi_n|}_{T}\equiv{\chi|}_T$.\par

\subsection{Group \str{$C^*$}{C*}-algebras.}Let $G$ be a discrete group. We define the {\em reduced $C^*$-algebra of $G$} by
\[C_{\lambda}^*(G):=\ov{\lambda_{\ell^1(G)}(\ell^1(G))}^{\|\cdot\|_2}\]
where $\lambda_{\ell^1(G)}$ is the $\ell^1(G)$-representation associated to $\lambda_G:G\rightarrow \bsf(\ell^2(G))$ by setting $\lambda_G(s)f(t)=f(s^{-1}t)$ for all $s\in G$. If instead we close the set $\ell^1(G)$ via 
\[\|f\|_u=\sup\l\{\|\pi(f)\|\,\colon\,\pi\text{ is a *-representation of }\ell^1(G)\r\},\]
then we have defined the {\em full group $C^*$-algebra of $G$}, denoted $C^*(G)$.
When $G$ is {\em amenable}, $C^*(G)$ is isomorphic to $C^*_{\lambda}(G)$. See \cite[Ch. VII]{Dav96} or \cite[13.9 (p.303)]{Dix77} for a more in-depth discussion of this construction.

Except for degeneracy, all the notions of representations for $C^*$-algebras are analogous to those of unitary representations of groups. We use $\usf(\mc{H})$ to denote the group of unitary operators on a Hilbert space, $\mc{H}$. The set of equivalence classes of all irreducible unitary representations of $G$, denoted by $\wh{G}$, is called the {\em unitary dual} of $G$. We observe that when $G$ is discrete and abelian, the unitary dual is in bijection with the Pontryagin dual.

Every irreducible representation of $C^*(G)$ is in a dimension preserving one-to-one correspondence with irreducible unitary representations of $G$ \cite[Ch. VII]{Dav96}. Thus, there is an intimate connection between the spectrum of $C^*(G)$ and the unitary dual of $G$. Except in the case when $G$ is abelian (and, thus, the unitary dual of $G$ may be viewed as its Pontryagin dual), $\wh{G}$ does not possess a native topology. Using the topology on $\wh{C^*(G)}$, we topologize $\wh{G}$ for any discrete group and observe that, in the case $G$ is abelian, $C^*(G)\cong C_{\lambda}^*(G)\cong C(\widehat{G})$. Hence, for any discrete group $G$, $\wh{C^*(G)}$ is homeomorphic to $\wh{G}$, which we denote $\wh{C^*(G)}\cong \wh{G}$. Finally, we note that $\te{Prim}_n(C^*(G))\cong \wh{G}_n$ for each $n$ where $\wh{G}_n$ is the space of $n$-dimensional irreducible unitary representations of $G$.

\subsection{Virtually abelian groups.} A group $G$ is \emph{virtually abelian} if there exists an abelian subgroup $H$ of finite index. In the case $G$ is finitely generated, then so is $H$. Hence, $H$ contains a torsion-free subgroup of finite index, say $K\cong\Z^r$. By taking normal core of $K$ in $G$, there exists $N\unlhd G$ such that $[G:N]<\infty$ and $N\leq K$. Because $N$ has finite index in $K\cong \Z^r$, it follows that $N\cong \Z^r$. We gather the above observations into the following remark.
\begin{rmk}\label{virtually_abelian_general_form}
    A group $G$ is finitely generated and virtually abelian if and only if $G$ fits into a short exact sequence of the form
    \begin{equation}\label{eqn:virtually_abelian_defn}
         1\rightarrow \Z^r\stackrel{i}{\rightarrow} G\stackrel{s}{\rightarrow}D\rightarrow 1
    \end{equation}
    with $|D|<\infty$. 
\end{rmk}

The number $r$ above is the {\em rank of $G$}. It is also called the \emph{Hirsch length} (we write $h(G)=r$). In fact, the Hirsch length is well-defined for every virtually abelian group (and more generally for every elementary amenable group). For more information regarding the Hirsch length and why it is well-defined, we refer the reader to \cite{hillman_finite_Hirsch_length_old}. \par
Let $G$ be virtually abelian and identify $i(\zn^r)\unlhd  G$ with $\zn^r$ where we treat $\zn^r$ as a multiplicative group. Because $\zn^r$ is normal in $G$, there is a natural action of $G$ on $\Z^r$ defined by $g\cdot a=gag^{-1}$ for all $g\in G$, $a\in \Z^r$. Let $\gamma:D\rightarrow G$ be a section with $\gamma(1_D)=1_G$. Then, the action of $G$ on $\Z^r$ $(G\curvearrowright \Z^r$) descends to an action of $D$ on $\Z^r$ by $d\cdot a=\gamma(d)\cdot a$. Notice the induced action is independent of the section we choose. 

We also have an induced (left) action $G\curvearrowright\wh{\Z}^r$ given by 
\[(g\cdot \chi)(a)=\chi(g^{-1}ag)\text{ for all }g\in G, \chi\in\wh{\Z}^r, a\in \Z^r.\]
This action descends to an action of $D$ on $\widehat{\Z}^r$. For each $\chi\in \wh{\Z}^r$, we define
\begin{equation}\label{eqn:stabilizer_and_orbit}
G_{\chi}=\{g\in G\,\colon\,g\cdot \chi=\chi\}\quad\te{ and }\quad\mc{O}_{\chi}=\{g\cdot\chi\,\colon\,g\in G\}
\end{equation}
to be the stabilizer subgroup associated to $\chi$ and the orbit associated to $\chi$, respectively. We observe that $|\mc{O}_{\chi}|={|G/ G_{\chi}|}$, $\Z^r\leq G_{\chi}$, and $|\mc{O}_{\chi}|$ divides $|D|$ for all $\chi\in\wh{\Z}^r$. \par

\begin{thm}[\cite{Moo72}, \cite{Dix77}]\label{note:subhomogeneous}
  Let $G$ be a discrete group. Then $C^*(G)$ is separable and subhomogeneous if and only if $G$ is a countable, virtually abelian group. 
\end{thm}

When $G$ is finitely generated and virtually abelian, $\wh{G}\cong\widehat{C^*(G)}\cong$ \par\noindent$ \Prim(C^*(G))$. Throughout the paper, we will use $\wh{G}$, $\widehat{C^*(G)}$, and $\Prim(C^*(G))$ interchangeably.

\subsection{Covering dimension.}
In this subsection, we present some results on covering dimension which will be used in the sequel. For a definition and important properties, we refer the reader to \cite{Pea75}. Recall that a topological space $X$ is called \emph{totally normal ($T_5$)} if every subspace of $X$ is normal.
\begin{prop}[Theorem 6.4. \cite{Pea75}]\label{monot_covering_dim}
    Let $X$ be a totally normal space and $Y\subseteq X$. Then $\dim(Y)\leq \dim(X)$.
\end{prop}
\begin{prop}[Chapter 9, Proposition 2.16, \cite{Pea75}]\label{dim_of_quotient}
     Let $X$, $Y$ be paracompact, normal topological spaces and $f:X\rightarrow Y$ be a continuous open surjection such that $f^{-1}(y)$ is finite for every $y\in Y$. Then $\dim(X)=\dim(Y)$.
 \end{prop}
 
The following result is known to experts, but we present a proof for the sake of completion.

 \begin{lemma}\label{non_empty_int_and_dimension}
     Suppose $X= \TT^n\times F$ for some $n\in \N$ where $\TT^n$ and $F$ a finite set. We endow $\TT^n$ $F$ with the discrete metric and $X$ with the product metric. If $U\subseteq X$ has non-empty interior, then $\dim(U)=n$.
     \begin{proof}
         Let $x\in U^{\circ}.$ Then there exists $\varepsilon>0$ such that $B(x,\varepsilon)\subseteq U \subseteq X.$ But for small enough $\varepsilon$, $B(x,\varepsilon)$ is homeomorphic to the unit ball (in $r$-dimensions). So, $\dim(B(x,\varepsilon))=r$. The result follows from the fact that $X$ is a metric space (hence totally normal) and Proposition \ref{monot_covering_dim}.
     \end{proof}
 \end{lemma}

 \subsection{Nuclear dimension.}
The notion of the {\em nuclear dimension} was introduced by Winter and Zacharias in \cite{WZ10}. In that paper, they showed that $\dimnuc (C(X))=\dim(X)$ for every locally compact second countable Hausdorff space $X$. In this sense, nuclear dimension can be viewed as a non-commutative analog of the covering dimension. 

We refer the reader to \cite{WZ10} for the precise definition and basic properties of nuclear dimension. \par
In this paper, we are interested in computing the nuclear dimension on the setting of subhomogeneous $C^*$-algebras. For such $C^*$-algebras, Winter has shown that it is connected with the dimensions of the spaces of $k$-dimensional irreducible representations.
\begin{thm}[cf. Main Theorem, \cite{Win04}]\label{dim_nuc_subhomogeneous}
 Let $A$ be a separable subhomogeneous $C^*$-algebra. Then 
\[\dimnuc (A)=\max_{i\in \N}\{\dim \textnormal{Prim}_i(A)\}.\]
 \end{thm}
We remark the statement of the Main Theorem in \cite{Win04} is slightly different than presented here. For the exact statement, see \cite[Thm. 2.6]{BL24}). \par

It is already known that $\dimnuc (C^*(G))\leq h(G)$ for every finitely generated virtually abelian group $G$ \cite[Prop. 2.14]{BL24}.\footnote{Actually, this result is stated in terms of the asymptotic dimension, $\te{asdim}\,(G)$. However, $\te{asdim}\,(G)=h(G)$ for every finitely generated virtually abelian group $G$ by \cite[Thm. 3.5]{asymp_dim_va_groups}.} Our main result (Theorem \ref{main_result_2}) shows that equality holds.

\section{Results on Orbits and Stabilizers of Virtually abelian groups}
\label{sec:orbs-stabs}

In this section, we fix a finitely generated virtually abelian group $G$ of rank $r$. Using the extension from \eqref{eqn:virtually_abelian_defn} as the framework, we conflate $\Z^r$ with the image $i(\Z^r)\leq G$; similarly, we conflate $D$ with the finite quotient group $G/\Z^r$, and take $s\colon G\to D$ as the natural projection map. We focus on the centralizer subgroup $L:=C_G(\Z^r)$ for constructing a topological space $N_{\mt{K}}/D_1$ of dimension $r$. This subgroup $L$ will then act as our primary foothold to connect the nuclear dimension of $C^*(G)$ with the Hirsch length $h(G)=r$. 

\subsection{Centralizer of $\zn^r$ in $G$.}\label{subsec:centralizer} The conjugation action $G\curvearrowright \Z^r$ admits the centralizer subgroup $L$ as its kernel; thus, $L$ is a finite index normal subgroup of $G$ containing $\Z^r$.

Define the finite quotient groups $D_0:=L/\Z^r$ and $D_1:=G/L$, denote by $i$ the inclusion maps, and let $s_0$, $s_1$ be the natural projection maps provided in the following exact sequences
\begin{equation}
    \label{eqn:defining_L}
     1\rightarrow \Z^r\stackrel{i}{\rightarrow} L\stackrel{s_0}{\rightarrow}D_0\rightarrow 1\quad\te{ and }\quad     1\rightarrow L\stackrel{i}{\rightarrow} G\stackrel{s_1}{\rightarrow}D_1\rightarrow 1.
\end{equation}

\noindent Furthermore, we set $\mt{K}:=|D_1|$ and define
\[\wh{L}_{1D}:=\te{Hom}\,(L,\T)\] 
as the subspace of the 1-dimensional representations (or characters) of $L$. Notice that $\wh{L}_{1D}\cong \widehat{L_{ab}}$ where $L_{ab}=L/[L,L]$ is the
quotient by the commutator subgroup $[L,L]$ of $L$. Endowing $\wh{L_{ab}}$ with the metric induced by the topology of pointwise convergence, we can view $\wh{L}_{1D}$ as a compact metric space.  \par

The first extension in (\ref{eqn:defining_L}) is a central extension, which implies $L$ is a BFC group.\footnote{BFC stands for \textit{boundedly finite class of conjugate elements}.} That is, there exists $d\in\zn_{>0}$ such that no element of $L$ has more than $d$ conjugates. Indeed, observe that $\zn^r\leq C_G(x)$, for each $x\in L$. By the orbit-stabilizer theorem, the size of the conjugacy class of $x$ is $[G:C_G(x)]\leq [G:\Z^r]=|D|$.\par
 Since $L$ is a BFC group, a result of B. H. Neumann (see for example \cite[14.5.11]{robinson_group_theory_book}) implies that $[L,L]$ is finite.

\begin{ex}\label{example:crystallographic_group_defn}
Notice that the action $G\curvearrowright \Z^r$ is faithful if and only if $L=C_G(\Z^r)=\Z^r$ if and only if $\Z^r$ is maximally abelian in $G$. If any of these equivalent conditions hold, we say that $G$ is a {\em crystallographic group} of dimension $r$. This class of groups is a well-studied object and is of independent interest to the fields of physics and chemistry. Crystallographic groups include the 17 two-dimensional wallpaper groups and the 230 three-dimensional space groups (219 up to abstract group isomorphism). See \cite{Hil86} for an elementary mathematical introduction. \par
\end{ex}

 \subsection{Extension of characters.}\label{subsec:extension_of_characters}
We continue the section with a result on extension of characters. In particular, we will show that every character of $\Z^r$ extends to a character of $L$. \par

\begin{lemma}\label{basic_lemma_11}
     Let $\{e_1,...,e_r\}$ be a $\Z$-basis of $\Z^r$, treated as a multiplicative group. For each $x\in\Z^r$, denote with $\bar{x}$ the image of $x\in \Z^r\leq L$ onto $L_{ab}$. Then $\{\bar{e_1},...,\bar{e_r}\}$ is $\Z$-linearly independent in $L_{ab}$.
     \begin{proof}
         Take integers $a_1,...,a_r$ such that $\bar{x}=\bar{e_1}^{a_1}\cdots\bar{e_r}^{a_r}=1_{L_{ab}}$, i.e., $x=e_1^{a_1}\cdots e_r^{a_r}\in [L,L]$. Since $|[L,L]|<\infty$, $x$ has finite order. But $x$ is also an element of $\Z^r$, so it must be that $e_1^{a_1}\cdots e_r^{a_r}=x=1_{\zn^r}$. Hence, it follows that $a_1=a_2=\cdots=a_r=0$.
     \end{proof}
 \end{lemma}
 
\begin{lemma}\label{basic_lemma_12}
     Let $A$ be a finitely generated abelian group and $H\leq A$ a subgroup. Then every character $\chi\in \widehat{H}$ can be extended to a character $\widetilde{\chi}\in \widehat{A}$.
     \begin{proof}
        This follows from \cite[Cor. 3.6.2]{principles_of_harm_anal_book}. \qedhere
         
     \end{proof}
    
 \end{lemma}

\begin{prop}\label{extension_of_char}
    Every $\chi\in \widehat{\Z}^r$ can be extended to a character $\widetilde{\chi}\in \wh{L}_{1D}$.
    \begin{proof}
        Let $\chi\in \widehat{\Z}^r$ and fix $\{e_1,e_2,...,e_r\}$ as a $\Z$-basis of $\Z^r$. Let $H\leq L_{ab}$ be the subgroup generated by $\{\bar{e_1},\bar{e_2},...,\bar{e_r}\}$. Lemma \ref{basic_lemma_11} implies that $\{\bar{e_1},\bar{e_2},...,\bar{e_r}\}$ is $\Z$-linearly independent, and we see that $H$ has finite index in $L_{ab}$. Define
        \[\chi_H:H\rightarrow \TT\quad \te{via}\quad \chi_H(\bar{e_i})=\chi(e_i).\]
        Notice that $\chi_H$ is a character, so by Lemma \ref{basic_lemma_12} it can be extended to a character $\chi_{L_{ab}}:L_{ab}\rightarrow \TT$. Finally, $\chi_{L_{ab}}$ induces a map $\widetilde{\chi}\in \wh{L}_{1D}$. To finish the proof, observe $\widetilde{\chi}(e_i)=\chi_{L_{ab}}(\bar{e_i})=\chi(e_i)$, as desired.
    \end{proof}
\end{prop}

\subsection{Maximal orbits.} 
We now investigate the topology of the set of characters of $L$
with maximal orbits. To begin, we prove that all stabilizer subgroups of $G$ under the action $G\curvearrowright \Z^r$ contain $L=C_G(\zn^r)$ as a normal subgroup.

\begin{lemma}\label{stabilizers_in_virtually_abelian_case}
  Let $\psi\in \TT^r\cong \widehat{\Z^r}$. Then $G_{\psi}\geq L$ with equality if and only if the orbit $\mc{O}_{\psi}$ has order $\mt{K}$.
  \begin{proof}
      To show that $G_{\psi}\geq L$, we prove $g\cdot \psi=\psi$ for all $g\in L$.
      
     For any $g\in L=C_G(\zn^r)$ and $a\in \Z^r$,
      \[(g\cdot \psi)(a)=\psi(g^{-1}ag)=\psi(a).\]
     Thus, $g\in G_{\psi}$. \par
     Further,
    \[|\mc{O}_{\psi}|=[G:G_{\psi}]\leq [G:L]=|D_1|=\mt{K}\]
      So, 
      \[|\mc{O}_{\psi}|=\mt{K}\Longleftrightarrow [G:G_{\psi}]=[G:L]\Longleftrightarrow G_{\psi}=L.\qedhere\]
  \end{proof}
\end{lemma}

We now introduce the topological space which lies at the heart of our argument in Section \ref{sec:lower_bound}. Define $\rho(\chi)=\chi|_{\zn^r}$ for each $\chi\in \wh{L}_{1D}$. The maximal character space in $\wh{L}_{1D}$ is defined as
\begin{equation}\label{eqn:Nk}
N_{\mt{K}}:=\l\{\chi \in \wh{L}_{1D}\,\colon\, G_{\rho(\chi)}=L\r\}.
\end{equation}

\begin{lemma}\label{desired orbit is open}
    $N_{\mt{K}}$ is open in $\wh{L}_{1D}.$
    \begin{proof}
           It enough to show that $\wh{L}_{1D}\backslash N_{\mt{K}}$ is closed.
            Let $\chi_n\to \chi$ with $\chi_n\notin N_{\mt{K}}$. Then $G_{\rho(\chi_n)}\gneq L$ by Lemma \ref{stabilizers_in_virtually_abelian_case}. By \cite[Prop 4.12]{CW24} we have that $G_{\rho(\chi)}\gneq L$. Thus $\chi\notin N_{\mt{K}}$. 
    \end{proof}
\end{lemma}

We have an action $G\curvearrowright \widehat{L}_{1D}$ defined via $(g\cdot \chi)(a)=\chi(g^{-1}ag)$ for every $\chi\in \wh{L}_{1D}$, $g\in G$ and $a\in L$. For every $g\in L$, we have $(g\cdot \chi)(a)=\chi(g^{-1}ag)=\chi(g^{-1})\chi(a)\chi(g)=\chi(a)$. It follows that the action is trivial on $L$, and thus it induces an action $D_1\curvearrowright \wh{L}_{1D}$. Notice that this induced action can be defined via $(d_1\cdot \chi)(a)=\chi(\gamma_1(d_1)^{-1}a\gamma_1(d_1))$. Here, $\chi\in \wh{L}_{1D}$, $d_1\in D_1$, $a\in L$, and $\gamma_1:D_1\to G$ is any section. \par
Moreover, the action $D_1\curvearrowright \wh{L}_{1D}$ is an action by homeomorphisms.\footnote{This means that it induces a group homomorphism $D_1\to \operatorname{Homeo}(\wh{L}_{1D})$} Indeed, the action induces a continuous map $D_1\times \wh{L}_{1D}\to \wh{L}_{1D}$ via $(g,\chi)\to g\cdot \chi$. It follows that for every $g\in D_1$, the map $\chi\mapsto g\cdot \chi$ is a homeomorphism. \par
We turn our attention to the maximal orbit space of $\wh{L}_{1D}$, i.e., the quotient space $N_{\mt{K}}/D_1$.
We view each orbit as a single point in this quotient space. \par

\begin{rmk}\label{rmk:quotient_map}
Let $q:N_{\mt{K}}\rightarrow N_{\mt{K}}/D_1$ be the quotient map, which is continuous by definition of the quotient topology. We show $q$ is open. Indeed, let $U\subset N_{\mt{K}}$ be open. We observe that, for any $g\in D_1$, $g\cdot U$ is open because the action $D_1\curvearrowright \wh{L}_{1D}$ is by homeomorphisms (see comments below Lemma \ref{desired orbit is open}). Then $D_1\cdot U=\bigcup_{g\in D_1}g\cdot U$ is open as a finite union of open sets. Set $V=q(U)$ and note 
\[D_1\cdot U=\{\chi\in N_{\mt{K}}: q(\chi)\in V\}.\]
Because $D_1\cdot U$ is open and $q$ is a quotient map, $V$ is open. \par
Replacing $U$ by a closed set $F$, an identical argument implies that $q$ is also a closed map. 
\end{rmk}

\begin{lemma}\label{quot_is_hausdorff}
    $N_{\mt{K}}/D_1$ is Hausdorff.
    \begin{proof}
 We use the notation $\chi \sim \chi'$ if and only if $\chi$ and $\chi'$ are on the same orbit. Since $N_{\mt{K}}\sse \wh{L}_{1D}=\te{Hom}(L,\tn)$, $N_{\mt{K}}$ is Hausdorff, and since $N_{\mt{K}}\subset \wh{L}_{1D}\cong \wh{L}_{ab}$ is a metric space, so too is $N_{\mt{K}}\times N_{\mt{K}}$ with the product metric. By \cite[Ex. 2.4.C(c)]{engelking_topology}, it is enough to show that the set $\{(\chi,\psi)\in N_{\mt{K}}\times N_{\mt{K}}\,\colon\,\chi\sim \psi\}$ is closed in $N_{\mt{K}}\times N_{\mt{K}}$.\footnote{$N_{\mt{K}}\times N_{\mt{K}}$ is endowed with the product topology.} Assume that $(\chi_n,\psi_n)\in N_{\mt{K}}\times N_{\mt{K}}$ converges to $(\chi,\psi)$ where $\chi_n\sim \psi_n$ for all $n$. Then $\chi_n\to \chi$ and $\psi_n\to \psi$. Because $\chi_n$ and $\psi_n$ are on the same orbit, there exist $d_n\in D_1$ such that $\chi_n=d_n\cdot \psi_n$. Because $D_1$ is a finite group, we can assume, after passing to a subsequence, that $d_n=d$ for every $n$. Thus $\chi_n=d\cdot \psi_n$. By taking limits as $n\to \infty$ and using the above, we deduce that $\chi=d\cdot \psi$. So, $\chi\sim \psi$ and thus the proof is complete. 
    \end{proof}
\end{lemma}
Our next goal is to examine how ``large" $N_{\mt{K}}$ and $N_{\mt{K}}\slash D_1$ are, which we quantify by their covering dimension. This measurement will be used in Section \ref{sec:lower_bound}.  \par

We begin by showing that $N_{\mt{K}}$ is not empty. As we saw in Section \ref{subsec:extension_of_characters}, characters of $\TT^r$ always extend to characters of $L$. Hence, to prove $N_{\mt{K}}\neq\varnothing$, it is enough to show that there exists $\chi\in \TT^r$ with stabilizer equal to $L$ (equivalently with $\mt{K}$-orbit). Actually, we show that the characters with the above property are dense in $\TT^r$. The following result and its proof are very similar to \cite[Lemma 2.1]{strongly_qd_eckhardt}.

\begin{prop}\label{density}
    $M:=\{\chi\in \TT^r \,\colon\, G_{\chi}=L\}$ is dense in $\TT^r$.
    \begin{proof}
     
 For every $d\in D$, define
\[A_d:=\on{Fix}_{\T^r}(d)=\{\chi\in \TT^r\,\colon\, d\cdot \chi=\chi\}\]
where the action that is involved is $D\curvearrowright \TT^r$.
Recall that $D_0:=L\slash \zn^r=C_G(\zn^r)\slash \zn^r$. We will show that for every $d\in D\setminus D_0$, $A_d^{\circ}=\varnothing.$ We note that $A_d\neq \tn^r$ when $d\not\in D_0$. \par

For the sake of contradiction, suppose that $A_d^{\circ}\neq\varnothing$ for some $d\notin D_0$. Let $x\in A_d^\circ$ with $B(x,\varepsilon)\sse A_d$ and define $V=x^{-1}B(x,\varepsilon)$. Because $A_d$ is a subgroup of $\TT^r$, $1_{\TT^r}\in V\sse A_d$. Note that for any $y\in \TT^r$ the map $x\mapsto xy$ is an isometry. Then, a straightforward exercise in topological groups demonstrates that $\langle V\rangle$ is a clopen subgroup in $\T^r$. Because $\langle V\rangle \leq A_d$ and $\TT^r$ is connected, we get a contradiction.

So $A_d^{\circ}=\varnothing$ for every $d\notin D_0$.  But each $A_d$ is closed. Hence $\bigcup_{d\notin D_0}A_d$ also has empty interior. Because $M=\TT^r\backslash \bigcup_{d\notin D_0}A_d$, we deduce that $M$ is dense in $\TT^r$. \qedhere

    \end{proof}
\end{prop}

In order to compute the covering dimension of $N_{\mt{K}}/D_1$, we first compute the covering dimension of $N_{\mt{K}}$ and then apply Proposition \ref{dim_of_quotient} to pass to the quotient.

\begin{prop}\label{dim_N_k=r}
    $\dim(N_{\mt{K}}/D_1)=r$.
    \begin{proof}
        By Lemma \ref{desired orbit is open} and Proposition \ref{density}, $N_{\mt{K}}$ is open in $\wh{L}_{1D}$ and there exists $\chi\in \TT^r$ such that $G_{\chi}=L$.  Proposition \ref{extension_of_char} guarantees that $N_{\mt{K}}$ is non-empty. Moreover, $\wh{L}_{1D}\cong \widehat{L_{ab}}\cong \TT^r \times F$ for some finite set $F$ endowed with the discrete topology. It follows that $\wh{L}_{1D}$ is metrizable, whence totally normal. We conclude $\dim(N_{\mt{K}})=r$ via Lemma \ref{non_empty_int_and_dimension}. 
        
        Because $D_1$ is a finite group, $q^{-1}(y)$ is finite for every $y\in N_k/D_1$. Further, per Remark \ref{rmk:quotient_map}, $q$ is a continuous, open, and closed surjection. Since $N_{\mt{K}}$, $N_{\mt{K}}/D_1$ are normal and paracompact (\cite[1.5.20 and 5.1.33]{engelking_topology}), Proposition \ref{dim_of_quotient} implies that $\dim(N_{\mt{K}}/D_1)=\dim(N_{\mt{K}})=r.$
    \end{proof}
\end{prop}

\section{Proof of the main result}\label{sec:lower_bound}

We briefly provide a road map for our main result. We will first build off of the work in \cite{KanTay13,Mac58} to construct an injective map $\Phi: N_{\mt{K}}/D_1 \hookrightarrow \Prim_{\mt{K}}(C^*(G))$ for which $N_{\mt{K}}/D_1$ is homeomorphic to its image. This gives $\dim(\Prim_{\mt{K}}(C^*(G)))\geq \dim(N_{\mt{K}}/D_1)=h(G)$. The main theorem follows by combining the above with the known upper bound, $\dimnuc C^*(G)\leq h(G)$, from \cite[Prop. 2.14]{BL24} and the Main Theorem of \cite{Win04}.

\subsection{Defining \str{$\Phi$}{Phi}}

The arguments in this subsection rely on the Mackey Machine, which provides a complete description of $\wh{G}$ as a set. This construction is achieved via induced representations, which systematically extends representations from subgroups. We observe that induced representations are canonical in the sense of uniqueness up to orthonormal basis. See \cite[Ch 2]{KanTay13} for a more detailed description of this process.

\begin{thm}[Mackey Machine (\cite{KanTay13} Thm~4.28)]\label{thm:Mackey_machine}
Let $G$ be a discrete group containing a finite index normal abelian group $A$. Let
$\Omega\sse\wh{A}$ be a cross section of orbits under the action $G\curvearrowright \wh{A}$. Let $\wh{G}_{\chi}^{(\chi)}$ denote the subset of elements $\sigma\in\wh{G}_{\chi}$ where there exists $m\in\zn_{>0}$ such that
\begin{equation}\label{eqn:block_diag_restriction}
\sigma\big|_{A}=\chi^{\oplus m}.
\end{equation}
    Then 
    \[\wh{G}=\l\{\te{\normalfont{ind}}_{G_{\chi}}^G\sigma\,\colon\,\sigma\in \wh{G}_{\chi}^{(\chi)},\chi\in\Omega\r\}.\]
\end{thm}

We choose $G$ to be a finitely generated virtually abelian group of rank $r$ and define $L$, $\wh{L}_{1D}$, and $D_1$ as in Section \ref{sec:orbs-stabs} and $\rho$ and $N_{\mt{K}}$ as in Equation \ref{eqn:Nk}.

\begin{rmk}\label{rmk:orbits_same_induced}
We include three useful remarks.
\begin{itemize}
\item[(i)] Suppose $\chi$, $\chi'$ are in the same orbit. Then there exists $a\in G$ such that $\chi=a\cdot \chi'$. If we choose $\sigma\in\wh{G}_{\chi}^{(\chi)}$, then $\te{ind}_{G_{\chi}}^G\sigma\simeq \te{ind}_{G_{a\cdot \chi}}^Ga\cdot \sigma$ (\cite[Prop~2.39]{KanTay13}). This is to say, characters from the same orbit class induce the same representation. In addition, the Mackey Machine implies that whenever $\sigma\in G_{\chi}^{(\chi)}$, the induced representation, $\te{ind}_{G_{\chi}}^G\sigma$, is irreducible.
\item[(ii)] When $\chi\in N_{\mt K}$, then, by definition $G_{\rho(\chi)}=L$ where $\rho(\chi):=\chi|_{\Z^r}$. Thus, $\te{ind}_L^G\chi=\te{ind}_{G_{\rho(\chi)}}^G\,\chi$ is irreducible by the Mackey Machine.
\item[(iii)] \cite[Sec 4.1]{KanTay13} shows that each $\pi\in\wh{G}$ is associated with a unique orbit, which means there exists a unique $\mc{O}_{\zeta}$ (as in Eqn \ref{eqn:stabilizer_and_orbit}) for some $\zeta\in\wh{\zn}^r$ such that
\[\pi|_{\zn^r}=\l(\te{ind}_{G_{\zeta}}^G\sigma\r)\bigg|_{\zn^r}=\bigoplus_{g\in G\slash G_{\zeta}}(g\cdot \zeta)^{\oplus m}\]
where $\sigma\in\wh{G}_{\zeta}^{(\zeta)}$ and $m=\dim\sigma$. Moreover, if $\pi\simeq\pi'$ are associated with orbits $\mc{O}_{\zeta}$ and $\mc{O}_{\zeta'}$ then $\mc{O}_{\zeta}=\mc{O}_{\zeta'}$ In the case that $\chi\in \wh{L}_{1D}$ and $\te{ind}_L^G\chi$ is irreducible, we have $\sigma=\chi$, $G_{\rho(\chi)}=L$, and $|\mc{O}_{\rho(\chi)}|=\mt{K}$ (by Lemma \ref{stabilizers_in_virtually_abelian_case}).
\end{itemize}
\end{rmk}

 Define
\[\Phi:N_{\mt{K}}/D_1\rightarrow \te{Prim}_{\mt{K}}(C^*(G))\quad\te{via}\quad\Phi([\chi])=\te{ind}_L^G \chi.\]

We seek to prove that $\Phi$ is a homeomorphism onto its image (i.e., a well-defined, continuous, injective map, with continuous inverse). This claim relies on two lemmas. The first of which proves a more general version of injectivity for $\Phi$ using Remark \ref{rmk:orbits_same_induced} and the second is a standard exercise in topological sequences (quickly justified via contrapositive).

\begin{lemma}\label{lemma: when_ind_is_irreducible}
Let $\chi_1\in\wh{L}_{1D}$. Suppose there exists $\chi_2\in N_{\mt K}$ such that $\te{ind}_L^G\chi_2\simeq \te{ind}_L^G\chi_1$. Then, $\chi_1\in N_{\mt K}$ and $[\chi_1]=[\chi_2]\in N_{\mt K}/D_1$.

\begin{proof}
    Suppose $\te{ind}_L^G\chi_1\simeq \te{ind}_L^G\chi_2$ for $\chi_1\in\wh{L}_{1D}$ and $\chi_2\in N_{\mt{K}}$. Let $\rho(\chi_2)=(\chi_2)|_{\zn^r}$.

    First, we observe that $\chi_2\in\ N_{\mt K}$ and Remark \ref{rmk:orbits_same_induced} (ii) gives $\te{ind}_L^G\chi_2$ (and thus $\te{ind}_L^G\chi_1$) is irreducible. As indicated in Remark \ref{rmk:orbits_same_induced} (iii), there is a unique orbit  associated with $\te{ind}_L^G\chi_1\simeq\te{ind}_L^G\chi_2$. In particular, $\mc{O}_{\rho(\chi_1)}=\mc{O}_{\rho(\chi_2)}$ where $|\mc{O}_{\rho(\chi_2)}|=\mt{K}$. Therefore, Lemma \ref{stabilizers_in_virtually_abelian_case} implies $G_{\rho(\chi_1)}=L$ and $\chi_1\in N_{\mt K}$ by definition.

    Second, we note the dimension of $\te{ind}_L^G \chi_1$ is $[G:L]=|D_1|={\mt{K}}$. Thus, fix a Hilbert space $\mc{H}_{\mt{K}}$ as in Section \ref{rmk:standard_hilbert}. Because $\te{ind}_L^G\chi_1\simeq \te{ind}_L^G\chi_2$, there exists a unitary $U:\mc{H}_{\mt{K}}\rightarrow \mc{H}_{\mt{K}}$ such that 
     \[U\l[\l(\te{ind}_L^G\chi_1\r)(g)\r]U^{-1}(\xi)=\l[\l(\te{ind}_L^G\chi_2\r)(g)\r](\xi)\text{ for all }g\in G,\,\xi\in \mc{H}_{\mt{K}}.\]
     We observe that for any $\chi\in\wh{L}_{1D}$,
     \[\l[\te{ind}_L^G\chi\r](h)=\bigoplus_{a\in G\slash L}(a\cdot \chi)(h)=\bigoplus_{a\in D_1}(a\cdot \chi)(h)\text{ for any }h\in L\]
     because $L$ is normal in $G$. Therefore, for any $h\in L$ and $\xi\in \mc{H}_{\mt{K}}$,
     \begin{align*}
     U\l[\l(\te{ind}_L^G\chi_1\r)(h)\r]U^{-1}(\xi)&=\l[\l(\te{ind}_L^G\chi_2\r)(h)\r](\xi)\\
     U\l[\bigoplus_{a\in D_1}(a\cdot \chi_1)(h)\r]U^{-1}(\xi)&=\l[\bigoplus_{b\in D_1}(b\cdot \chi_2)(h)\r](\xi).\\
     \end{align*}
     
    Thus, for every $h\in L$, $\bigoplus_{a\in D_1}(a\cdot \chi_1)(h)$ and $\bigoplus_{b\in D_1}(b\cdot \chi_2)(h)$ are similar matrices and so they must have the same diagonal entries up to a permutation of $\{1,2,...,|D_1|\}$. Because the unitary $U$ that implements the similarity does not depend on $h$, neither does the permutation. We conclude that $[\chi_1]=[\chi_2]\in N_{\mt{K}}\slash D_1$. \qedhere
\end{proof}

\end{lemma}

\begin{lemma}\label{lem:convergence_equiv}
  Let $X$ be a topological space. With $(x_n)_{n\in \N},x\in X$, $x_n\to x$ if and only if for an arbitrary subsequence $(x_{k_n})_{n\in \N}$ there exists a subsequence $(\ell_n)_{n\in \N}\subset (k_n)_{n\in \N}$ such that $x_{\ell_n}\to x$.
\end{lemma}

\begin{prop}\label{phi_homeo_onto_image}
$\Phi$ is a homeomorphism onto its image.
\begin{proof}
We split the proof into four claims.

\begin{claim}
$\Phi$ is well-defined. 
\end{claim}

\begin{proof}[Proof of Claim 1]
 Fix $\chi\in N_{\mt{K}}$. By Remark \ref{rmk:orbits_same_induced} (ii), $\te{ind}_L^G\chi$ is irreducible. If $\chi_1,\chi_2\in N_{\mt K}$ are on the same orbit (under $D_1 \curvearrowright \wh{L}_{1D}$), Remark \ref{rmk:orbits_same_induced} (i) implies $\te{ind}_L^G \chi_1\simeq \te{ind}_L^G \chi_2$. \qedhere

 \end{proof}
 
\begin{claim}
$\Phi$ is continuous.
\end{claim}

\begin{proof}[Proof of Claim 2]
Because $q$ is a continuous open surjection (Remark \ref{rmk:quotient_map}), it follows that it is a quotient map. Let $\psi:N_{\mt K}\rightarrow \Prim_{\mt{K}}(C^*(G))$ be defined via $\psi(\chi)=\te{ind}_{L}^G\chi$ and observe that $\psi=\Phi \circ q$. By \cite[Thm 22.2]{Munkres_topology}, it is enough to show that $\psi$ is continuous. \par
         
         Let $\chi_n\to \chi$ in $N_{\mt K}$. \cite[Lemma 4.20]{CW24} and \cite[3.5.8 (p.83)]{Dix77} yield $\te{ind}_L^G\chi_n\to \te{ind}_L^G\chi$ in $\Prim_{\mt{K}}(C^*(G))$. \qedhere
\end{proof}
         
\begin{claim}
$\Phi$ is injective.
\end{claim}

\begin{proof}[Proof of Claim 3]
This follows by Lemma \ref{lemma: when_ind_is_irreducible}. \qedhere
\end{proof}

Because $\Phi$ is injective, $\Phi$ is bijective onto its image. Thus, there is a well-defined map 
\begin{align*}
\Psi:\Phi(N_{\mt K}/D_1)&\rightarrow N_{\mt K}\slash D_1\\
\Phi([\chi])&\mapsto[\chi]
\end{align*}

Although we have frequently abused notation and written $\pi\in\te{Prim}_*(C^*(G))$, we will need to be more precise in the proof of the following claim. We let $[[\sigma]]$ denote an element of $\te{Prim}_*(C^*(G))$ (which is an {\em equivalence class} of irreducible representations) and $\sigma$ to an element of $\Rep_*(C^*(G))$ (which is a representation on a fixed Hilbert space $\mc{H}_*$).

\begin{claim}
   $\Psi:\Phi(N_{\mt{K}}/D_1)\to N_{\mt{K}}/D_1$ is continuous.
\end{claim}
\begin{proof}[Proof of Claim 4]
Let $[[\pi_n]],[[\pi]] \in \Phi(N_{\mt{K}}/D_1)$ such that $[[\pi_n]]\to [[\pi]]$ in \par \noindent$\Prim_{\mt{K}}(C^*(G))$. Note that there exist $\chi_n,\chi\in N_{\mt{K}}\sse\wh{L}_{1D}$ such that $\pi_n\simeq \te{ind}_{L}^G \chi_n$ for every $n$, and $\pi\simeq \te{ind}_L^G \chi$. \par
Let $(k_n)_{n\in \N}$ be an increasing sequence of natural numbers. Because $\wh{L}_{1D}$ is compact (see discussion in Section \ref{subsec:centralizer}), there exists a subsequence $(\ell_n)_{n\in \N}\subset (k_n)_{n\in \N}$ and $\zeta\in \wh{L}_{1D}$ such that $\chi_{\ell_n}\to \zeta$ in $\wh{L}_{1D}$. Fix a Hilbert space $\mc{H}_{\mt K}$ as in Section \ref{rmk:standard_hilbert}. By \cite[Lemma 4.22]{CW24} we deduce that $\pi_{\ell_n}\to \te{ind}_L^G \zeta$ in $\Rep_{\mt{K}}(C^*(G))$ (where $\pi_{\ell_n}\in\te{Irr}_{\mt{K}}(C^*(G))$ for all $n\in\nn$). We show $\te{ind}_L^G\zeta\simeq\pi$.

Recall that $\te{Prim}_{\mt K}(C^*(G))$ is Hausdorff (see Remark \ref{rmk:prim_hausdorff}). Because we have assumed that $[[\pi_n]]\to [[\pi]]$ in $\Prim_{\mt{K}}(C^*(G))$, we must also have $[[\pi_{\ell_n}]]\rightarrow[[\pi]]$ where we note that this limit is unique. % The Hausdorffness of $\te{Rep}_{\mt{K}}(C^*(G))$ and that $\pi_{l_n}$ is $[[\pi_{\ell_n}]]$ represented on $\mc{H}_{\mt K}$ implies $\pi\simeq \te{ind}_L^G\zeta$.%
 But \cite[Thm. 4.7]{CW24}, implies that $[[\pi_{\ell_n}]]\to [[\te{ind}_L^G \zeta]].$ It follows that $\pi\simeq \te{ind}_L^G\zeta$.

Combining our assumptions with the previous conclusion, we see that $\te{ind}_L^G\zeta$ $\simeq \pi\simeq \te{ind}_L^G\chi$. Lemma \ref{lemma: when_ind_is_irreducible} then yields $[\zeta]=[\chi]$, and so $[\chi_{\ell_n}]\to [\chi]$. Lemma \ref{lem:convergence_equiv} implies $[\chi_n]\to [\chi]$, as desired. \qedhere

\end{proof}
The result follows immediately from the above four claims.

\end{proof}
\end{prop}

\begin{ex}\label{phi_not_surj}
    $\Phi$ need not be surjective. Indeed, let $D_0$ be any nonabelian finite group and $\pi\in \widehat{D_0}$ any irreducible representation such that $\dim(\pi)=\ell>1$. Consider the group $\zn^{\ell}\rtimes (D_0\times \zn_{\ell})$ where $D_0$ acts trivially on $\Z^{\ell}$ and $\Z_{\ell} \curvearrowright \Z^{\ell}$ via a cyclic automorphism of order $\ell$. Define $\widetilde{\pi}:=\pi\otimes \rho:D_0\times \Z_{\ell}\to \usf(\ell)$ to be the tensor product representation, which is irreducible (\cite[Ex. 2.36]{fulton_harris_rep_theory}). Here $\rho:\Z_{\ell} \to \TT$ can be taken to be any character of $\Z_{\ell}$. \par
     Let $\chi_0\in \TT^{\ell}$ be the trivial character over $\Z^{\ell}$ and define the irreducible representation \[\sigma:=\chi_0\times \widetilde{\pi}:G\to U(\ell)\] via $\sigma(g,d)=\widetilde{\pi}(d)$. Then, $\sigma(g)=I_{\ell}$ for every $g\in \Z^{\ell}$, which implies $\sigma|_{\zn^{\ell}}=\chi_0^{\oplus \ell}$. Because $G_{\chi_0}=G$, we deduce that $\sigma$ is not on the image if $\Phi$.

\end{ex}

\subsection{Topology of the Spectrum}

We now investigate the topology of the primitive spectrum of $C^*(G)$. In general, $\te{Prim}\,(C^*(G))$ is not Hausdorff (not even when $G$ is crystallographic, see \cite[Section 5]{CW24} for an explicit example). However, if we fix $k$ and restrict to $\te{Prim}_k(C^*(G))$, then the situation is much nicer. These topological spaces are not only Hausdorff, but even totally normal. \par

To this end, we provide two short lemmas which we expect are known to experts but we could not find explicitly in the literature.
\begin{lemma}\label{lem: closed restrictions}
    Let $X$ and $Y$ be topological spaces with $f:X\to Y$ a closed map. If $B\subseteq Y$ and we define $A:=f^{-1}(B)$, then the restriction and corestriction $f|_A^B: A\to B$ is closed. 
\end{lemma}

\begin{proof}
    Consider a closed set $C\subseteq A$. There exists a closed set $C_0\subseteq X$ such that $C=A\cap C_0$. It can be quickly checked that 
    \[ f(C)=B\cap f(C_0). \]
    In particular, $f(C)$ is the intersection of a closed set with $B$ and thus closed in $B$.
\end{proof}

\begin{lemma}\label{lemma:topology_lemma}
  Let $f:X\to Y$ be a continuous, closed and surjective map. Assume that $X$ is totally normal and $Y$ is Hausdorff. Then $Y$ is totally normal.
  \begin{proof}
    Let $Z\subseteq Y$ be a subspace. It is enough to show that $Z$ is normal. Let $W:=f^{-1}(Z)$ and $g:W\rightarrow Z$ be the restriction of $f$ to $W$. Via Lemma \ref{lem: closed restrictions}, $g$ is closed, continuous and surjective. Moreover, $W$ is normal as a subspace of a totally normal space. By \cite[Ex. 6, Section 31]{Munkres_topology}, $Z$ is normal, completing the proof.
  \end{proof}
\end{lemma}

In addition to the lemmas above, we will invoke a well known result of point set topology.

\begin{prop}[\cite{dieck08}, Prop 1.4.4]\label{prop:hausdorff_closed}
Let $f:X\rightarrow Y$ be a quotient map. If $X$ is a compact Hausdorff space, then the following are equivalent:
\begin{enumerate}[label=(\roman*)]
\item $Y$ is Hausdorff.
\item $f$ is a closed map.
\item $\ker(f):=\{(x,x')\in X\times X: f(x)=f(x')\}$ is a closed set in $X\times X$.
\end{enumerate}
\end{prop}

\begin{prop}\label{prim_k_comp_normal}
   For every $k\in \N$, $\textnormal{Prim}_{k}(C^*(G))$, endowed with the Fell topology, is a totally normal topological space.
   \begin{proof}
       We first consider the quotient map $$\rho:\Rep_k(C^*(G))\to \Rep_k(C^*(G))/\simeq$$ 
       where $\pi\simeq \sigma$ if and only if there exists a unitary $U$ such that $\pi(g) U=U\sigma(g)$ for all $g\in G$. We recall that $\Rep_k(C^*(G))$ is Hausdorff. The fact that $G$ is finitely presented,\footnote{It is known that if $[G:H]<\infty$ and $H$ is finitely presented, then $G$ is also finitely presented. Since finitely generated abelian groups are finitely presented, so are finitely generated virtually abelian groups.} \cite[Lemma 2.5]{CW24}, and $\usf(k)$ is compact imply that $\Rep_k(C^*(G))$ is compact. We will now show that $\ker(\rho)$ is closed. Indeed, let $(\pi_n,\sigma_n)\in \te{Rep}_k(C^*(G))\times\te{Rep}_k(C^*(G))$ converge to $(\pi,\sigma)$ where, for each $n\in\zn_{>0}$, $\pi_n\simeq \sigma_n$. Thus, for every $n$, there exist $U_n \in \usf({k})$ such that $U_n\pi_n(g)=\sigma_n(g)U_n$ for every $g\in G$. Because $\usf({k})$ is compact, there exists a subsequence $(w_n)_{n\in \N}$ and $U\in \usf({k})$ such that $U_{w_n}\to U$. By taking limits at infinity, we deduce that $U\pi(g)=\sigma(g)U$ for every $g\in G$. Hence $\pi\simeq \sigma$, verifying $\ker(\rho)$ is closed. \par
       
       Because $\rho$ is continuous and surjective, Proposition \ref{prop:hausdorff_closed} implies $\rho$ is closed. Since irreducibility is unitarily invariant, we have that $\rho^{-1}(\te{Irr}_k(C^*(G))/\simeq)=\te{Irr}_k(C^*(G))$. In particular, the restriction $\rho: \te{Irr}_k(C^*(G))\to \te{Irr}_k(C^*(G))/\simeq$ is a closed map by Lemma \ref{lem: closed restrictions}. \par
       
      $\te{Irr}_{k}(C^*(G))$ is totally normal (in fact, completely metrizable \cite[3.7.4 (p.89)]{Dix77}). So, by Lemma \ref{lemma:topology_lemma}, $\te{Irr}_k(C^*(G))/\simeq$ is also totally normal. Since $\te{Prim}_k(C^*(G))\cong\te{Irr}_k(C^*(G))/\simeq$ under the map $[\pi]\mapsto\ker \pi$, we conclude that $\te{Prim}_k(C^*(G))$ is totally normal. \qedhere
     
   \end{proof}
 \end{prop}

   Now we are ready to prove the main result of the paper.

 \begin{thm}\label{main_result_2}
Let $G$ be a discrete, finitely generated, virtually abelian group. Then $\dimnuc (C^*(G))=h(G).$
\begin{proof}
We first show that $\dim(\textnormal{Prim}_{\mt{K}}(C^*(G))\geq h(G)$. \par
Indeed, by Proposition \ref{phi_homeo_onto_image} we can view $N_{\mt{K}}/D_1$ as a subspace of $\te{Prim}_{\mt{K}}(C^*(G))$. So, Propositions \ref{dim_N_k=r}, \ref{prim_k_comp_normal}, and \ref{monot_covering_dim} imply that $h(G)=\dim(N_{\mt{K}}/D_1)\leq $ \par\noindent$\dim(\te{Prim}_{\mt{K}}(C^*(G)))$. \par
Combined with Theorem \ref{dim_nuc_subhomogeneous} and \cite[Prop. 2.14]{BL24}, we have the following series of inequalities
$$h(G)\leq \dim \te{Prim}_{\mt{K}}(C^*(G))\leq \dimnuc (C^*(G))\leq h(G).$$
Hence, equality must hold everywhere, so the result follows.
\end{proof}
 \end{thm}
 \begin{rmk}
     Although $\Phi$ may not be surjective (see Example \ref{phi_not_surj}), the proof of the above theorem tells us that $\dim(N_{\mt{K}}/D_1)=\dim(\te{Prim}_{\mt{K}}(C^*(G)))$.
 \end{rmk}

 \section{Crystal-Like Sequences}\label{sec:crystal_like_groups}

As noted in Example \ref{example:crystallographic_group_defn}, a well trodden family of virtually abelian groups is the crystallographic groups. These groups carry a faithful action $G\curvearrowright \Z^r$. Additionally, all crystallographic groups possess $\chi\in \T^r\cong \widehat{\Z^r}$ such that $|\mc{O}_\chi|=[G:\zn^r]$ (see Example \ref{ex:crystallographic_phi_surj}). In this section, we investigate an intermediary between virtually abelian and crystallographic groups which we coin \textit{crystal-like}. We will highlight some difficulties that arise when trying to classify crystal-like groups and close by demonstrating that the order of the point group is invariant for crystal-like group-lattice pairs. In particular, for crystallographic groups, the order of the point group is a $C^*$-invariant.
\par

\begin{defn}
    A \textit{group-lattice pair} is the data of groups $(G,A)$ with $A$ finitely generated, abelian, $A\trianglelefteq G$, and $[G:A]<\infty$. We call $A$ the \textit{lattice} and $D:=G/A$ the \textit{point group} for the group-lattice pair.
\end{defn}

\begin{defn}\label{defn:crystal-like}
    We say that a group-lattice pair, $(G,A)$, is \textit{crystal-like} if there exists $\chi \in \widehat{A}$ such that $|\mc{O}_{\chi}|=[G:A]$. The orbit here is taken under the (induced) conjugation action $ (G/A)\curvearrowright \widehat{A}$.
\end{defn}
For any $\chi\in \wh{A}$, if $|\mc{O}_{\chi}|=[G:A]$, we call the orbit \emph{principal}. 

\begin{rmk}\label{rmk:char_of_cryst_like}
We note that if $G_{\chi}=A$ for some $\chi\in\wh{A}$, then $|G\slash A|=|G\slash G_{\chi}|=|\mc{O}_{\chi}|$. Thus, $(G\slash A)\curvearrowright \wh{A}$ having a principal orbit is equivalent to there existing $\chi\in\wh{A}$ such that $G_{\chi}=A$. 

Let $D=G\slash A$. Notice that if $D\curvearrowright A$ is not faithful, then there exists $d\in D\setminus\{1\}$ such that $d\cdot a=a$ for every $a\in A$. It follows that $d^{-1}\cdot \chi=\chi$ for every $\chi\in \widehat{A}$ and thus $D\curvearrowright \widehat{A}$ does not have any principal orbit. In other words, if $D\curvearrowright \widehat{A}$ has a principal orbit, then $D\curvearrowright A$ is faithful.

The above discussion in combination with Lemma \ref{stabilizers_in_virtually_abelian_case} implies that $(G,A)$ is a crystal-like group-lattice pair if and only if $A$ is a stabilizer subgroup under the conjugation action $G\curvearrowright \wh{A}$.
\end{rmk}

\begin{ex}\label{ex:not_crystal_like}
Unfortunately, a faithful action $D\curvearrowright G$ does not necessarily give rise to a crystal-like group-lattice pair. Let $F$ be a finite abelian group such that $|F|<|\te{Aut}\,(F)|$ (e.g., $F=\l(\zn_2\r)^2$).
Then, consider the group $G$ fitting into the exact sequence
\[1\rightarrow \zn\times F\rightarrow G\rightarrow \te{Aut}\,(F)\rightarrow 1\]
    where for all $\sigma\in\te{Aut}\,F$, $(z,f)\in \zn\times F$, we define $\sigma\cdot (z,f)=(z,\sigma(f))$. This action is faithful but, for any $\chi\in \wh{\zn}\times \wh{F}$, $|\mc{O}_{\chi}|\leq |F|<|\te{Aut}\,(F)|$.
\end{ex}

\begin{prop}\label{prop:implies_crystal-like}
Let $G$ be a virtually abelian group as in Remark \ref{virtually_abelian_general_form}. If $L:=C_G(\Z^r)$ is abelian, then $(G,L)$ is a crystal-like group-lattice pair. 
\end{prop}

\begin{proof}
Throughout this proof $\wh{L}_{1D}$ and $N_{\mt{K}}$ are as in Section \ref{sec:orbs-stabs}. Because of Remark \ref{rmk:char_of_cryst_like}, it is enough to find $\chi\in \widehat{L}=\wh{L}_{1D}$ such that $G_{\chi}=L$. But every $\chi\in N_{\mt{K}}$ satisfies the above. Indeed, since $L$ is abelian, we have $L\leq G_{\chi}\leq G_{\rho(\chi)}=L$ for any $\chi\in N_{\mt{K}}$. This establishes that $G_{\chi}=L$. The set $N_{\mt K}$ is non-empty by Propositions \ref{density} and \ref{extension_of_char} so the proof is complete. \qedhere
\par

\end{proof}

Proposition \ref{prop:implies_crystal-like} and Example \ref{example:crystallographic_group_defn} implies that all crystallographic groups $G$ form crystal-like group-lattices $(G,\zn^r)$. Notice that in this case $L=C_G(\zn^r)=\Z^r$.

\begin{ex}

    Not all crystal-like group-lattice pairs arise from abelian centralizers. Let $H:=(\Z_n)^n$ for $n\geq3$ and consider the action $S_n\curvearrowright(\Z^r\times H)$ which is trivial on $\Z^r$ and where $\sigma\in S_n$ takes the $i^{\te{th}}$ coordinate to the $\sigma(i)^{\te{th}}$ coordinate on the elements of $H$. Using the induced semi-direct product $G:=(\Z^r\times H)\rtimes S_n$,  consider the group-lattice pair $(G,\ \Z^r\times H)$. By construction of the action, $(\Z^r\times1_H)\rtimes S_n\leq C_G(\Z^r\times1_H)$, so the centralizer is a non-abelian group.

    Under the identification $\wh{H}=H$, the character corresponding to the tuple $h:=(0,1,\dots,n-1)\in H$ is fixed only by $1_{S_n}$, so this character represents a principal orbit.
\end{ex}

Our next goal is to show that the representation theory of groups arising from crystal-like group lattices $(G,A)$ remembers $[G:A]$. We require a few initial results.

The following is a translation of \cite[Cor~7.15(3)]{CST22} which is justified by \cite[Prop~3.22]{CW24}.

\begin{prop}\label{prop:sum_dim_fixed_char}
    Let $(G,A)$ be a group-lattice pair. Fix $\chi\in \wh{A}$ and let $\widehat{G}^{(\chi)}_{\chi}=\{\sigma_1,...,\sigma_{\ell}\}$\footnote{$\wh{G}_{\chi}^{(\chi)}$ is a necessarily finite set as it is in bijection with irreducible representations arising from a finite group | see \cite[Prop 7.12]{CST22} and \cite[Thm 3.16]{CW24}.} (as in Theorem \ref{thm:Mackey_machine}). Then $\ds\sum_{i=1}^{\ell}(\dim\sigma_i)^2=[G_{\chi}:A]$.
\end{prop}

\begin{lemma}\label{lemma:dim_argument}
Let $(G,A)$ be a group-lattice pair. If $\pi\in\wh{G}$, then $\dim \pi\leq [G:A]$. Moreover, if $\dim\pi=[G:A]$ is maximized, then by setting $\pi=\te{\normalfont{ind}}_{G_{\chi}}^G\,\sigma$, where $\sigma\in\wh{G}_{\chi}^{(\chi)}$, we have $\dim\sigma=1$, $G_\chi=A$, and $\sigma=\chi\in\wh{A}$.

%Moreover, given $\sigma\in\wh{G}_{\chi}^{(\chi)}$ with $\chi\in\wh{A}$, the condition $\dim\te{\normalfont{ind}}_{G_{\chi}}^G\,\sigma=[G:A]$ implies that $\dim \sigma=1=[G_{\chi}:A]$, and in particular, $\sigma=\chi$.
\end{lemma}

\begin{proof}
Let $\pi\in\wh{G}$. By the Mackey Machine, there exists $\chi\in\wh{A}$ and $\sigma\in\wh{G}_{\chi}^{(\chi)}$ such that $\pi=\te{\normalfont{ind}}_{G_{\chi}}^G\,\sigma$. Proposition \ref{prop:sum_dim_fixed_char} gives $\dim \sigma\leq [G_{\chi}:A]$. Thus, 
\[\dim \pi=\dim\te{ind}_{G_{\chi}}^G\,\sigma\leq [G:G_{\chi}]\,[G_{\chi}:A]=[G:A].\]

Suppose $\dim\te{ind}_{G_{\chi}}^G\,\sigma=[G:A]$.  Write
\begin{align*}
[G:A]&=\dim\te{ind}_{G_{\chi}}^G\,\sigma=[G:G_{\chi}]\,\dim\sigma\\
\Rightarrow \hspace{.15in} [G_{\chi}:A]&=\dim\sigma
\end{align*}
Again, by Proposition \ref{prop:sum_dim_fixed_char}, if $\wh{G}_{\chi}^{(\chi)}=\{\sigma_1,\dots,\sigma_{\ell}\}$ (where we assume, WLOG, $\sigma=\sigma_1$), then 
\[\sum_{i=1}^{\ell}(\dim \sigma_i)^2=[G_{\chi}:A].\]
Because $\dim\sigma_1=\dim\sigma=[G_{\chi}:A]$, we must have
\[[G_{\chi}:A]=\dim\sigma\leq(\dim \sigma)^2\leq [G_{\chi}:A].\]
As $\dim\sigma=(\dim\sigma)^2$, we conclude $\dim\sigma=1$. By definition, $\sigma\in \wh{G}_{\chi}^{(\chi)}$ means $\sigma|_A$ is a multiple of $\chi$. Since $1=[G_{\chi}:A]$, we have $G_{\chi}=A$ and so we have $\sigma|_A=\sigma=\chi$.
\end{proof}

\begin{prop}\label{prop:equivalence_faithful_action}
Suppose $(G,A)$ is a group-lattice pair. There exists $\pi\in\wh{G}$ with dimension equal to $[G:A]$ if and only if $(G\slash A)\curvearrowright \widehat{A}$ has a principal orbit.
\end{prop}

\begin{proof}

($\Rightarrow$) Suppose there exists $\pi\in\wh{G}$ with $\dim \pi=[G:A]$. Then there exists $\chi\in\wh{A}$ and $\sigma\in\wh{G}_{\chi}^{(\chi)}$ such that $\pi=\te{ind}_{G_{\chi}}^G\,\sigma$. Because $[G:A]=\dim \pi$, Lemma \ref{lemma:dim_argument} implies  $G_{\chi}=A$. So there exists a principal orbit.

($\Leftarrow$) Assume that there exists $\chi\in\wh{A}$ with $|\mc{O}_{\chi}|=[G:A]$. We see that $G_{\chi}=A$ and so $\te{ind}_A^G \chi$ is an irreducible representation. Moreover, 
\begin{align*}
\dim\te{ind}_{G_{\chi}}^G\chi&=\dim\te{ind}_A^G\,\chi\\
&=[G:A]\cdot\dim\chi\\
&=[G:A].\qedhere
\end{align*}
\end{proof}

\begin{cor}\label{cor:max_dim}
 Suppose $(G,A)$ is a group-lattice pair. Then $(G,A)$ is crystal-like if and only if $$\max\{\dim\pi\,\colon\,\pi\in \widehat{G}\}=[G:A].$$
\end{cor}

It is possible for a crystal-like group to have two decompositions satisfying the assumptions of Definition \ref{defn:crystal-like}. However, the index $[G:A]$ is recovered.

\begin{cor}\label{cor:point_group_order}
Let $G$ be a group with group-lattice pairs $(G,A_1)$ and $(G,A_2)$. If both $(G,A_1)$ and $(G,A_2)$ are crystal-like, then $[G:A_1]=[G:A_2]$.
\end{cor}

We end the section by noticing that Lemma \ref{lemma:dim_argument} implies that the map $\Phi$ defined in Section \ref{sec:lower_bound}, is an isomorphism for every crystallographic group.

\begin{ex}\label{ex:crystallographic_phi_surj}
    (All the notation is as in Section \ref{sec:lower_bound}). Let $G$ be a crystallographic group. Then $L=C_G(\zn^r)=\Z^r$ (see Example \ref{example:crystallographic_group_defn}) and $\mt{K}=[G:L]$. Assume that $\pi\in\wh{G}$ with dimension $\mt{K}$. Lemma \ref{lemma:dim_argument}, along with the Mackey Machine (Theorem \ref{thm:Mackey_machine}), imply that $\pi=\te{ind}_{\Z^r}^G \chi$ for some $\chi\in \TT^r=\wh{L}_{1D}$ with $G_{\chi}=L$. Thus, $\Phi$ is surjective, and hence an isomorphism.
\end{ex}

\begin{cor}\label{crystal_like_order_of_point_group}
 Let $G,H$ be discrete, finitely generated groups such that $C^*(G)\cong C^*(H)$. If $(G,A_G)$ and $(H,A_H)$ are crystal-like group-lattice pairs, then $[G:A_G]=[H:A_H]$.
 \end{cor}

 \begin{proof}
Since $C^*(G)\cong C^*(H)$, we also have $\wh{G}\cong \wh{H}$. By Corollary \ref{cor:max_dim}, 
\[[G:A_G]=\max\{\dim\pi_G\,\colon\, \pi_G\in \widehat{G}\}=\max\{\dim\pi_H\,\colon\, \pi_H\in \widehat{H}\}=[H:A_H].\qedhere\]
 \end{proof}

\begin{ex}
A group $C^*$-algebra $C^*(G)$ arising from groups $G$ which admit crystal-like group-lattices $(G,A)$ need not recover the isomorphism class of $A$. Consider 
\[G_1=(\Z^r\times\zn_2\times\zn_2)\rtimes_{\alpha}\zn_2 \quad\te{ and }\quad G_2=(\Z^r\times\zn_4)\rtimes_{\beta}\zn_2\]
with $\alpha(z,a,b)=(z,b,a)$ and $\beta(z,x)=(z,x^{-1})$. $G_1$ and $G_2$ are not isomorphic because $G_1$ does not contain an element of order 4 while $G_2$ does. We show that $C^*(G_1)$ is $C^*$-isomorphic to $C^*(G_2)$.

We first observe that, because $G_1,G_2$ are semidirect products by finite groups, their group $C^*$-algebras are crossed-products (see \cite{phi17} for a comprehensive introduction). We then have $C^*$-isomorphisms
\[C^*(G_1)\cong \left( C^*(\Z^r)\otimes C^*(\Z_2\times\Z_2)\right)\rtimes_{\hat{\alpha}} \Z_2\quad C^*(G_2)\cong \left( C^*(\Z^r)\otimes C^*(\Z_4)\right)\rtimes_{\hat{\beta}} \Z_2. \]

We recall the well-known $C^*$-isomorphism $\theta:C^*(\Z_2\times \Z_2)\to C^*(\Z_4)$ defined by $\theta(a)= \frac{x-ix^3}{1-i}$ and $\theta(b)=\frac{x^3-ix}{1-i}$ for $a$ and $b$ the generating unitaries of order $2$ and $x$ the generating unitary of order $4$. Noticing that $\theta \hat \alpha=\hat \beta \theta$, a standard argument utilizing the universal property of crossed product $C^*$-algebras provides the $C^*$-isomorphism $C^*(G_1)\cong C^*(G_2)$.
\end{ex}

\section{$C^*$-(super)rigidity properties for crystallographic groups}
 
 In this section, we refine our gaze to crystallographic groups. Notice that Theorem \ref{main_result_2} together with Theorem \ref{note:subhomogeneous} imply that if $G$ is a finitely generated virtually abelian group and $H$ a finitely generated discrete group such that $C^*(G)\cong C^*(H)$, then $H$ is also virtually abelian and $h(G)=h(H)$. In other words, the Hirsch length of a finitely generated, virtually abelian group is recovered by its group $C^*$-algebra. Within the context of crystallographic groups, this means that the dimension is $C^*$-invariant.
 
The primary goal of this section is to show that ``crystallographic" is a property remembered by a group's associated $C^*$-algebra (see Theorem \ref{crystallographic_is_preserved} for the exact statement). Once that is accomplished, Section \ref{sec:crystal_like_groups} may be invoked with the $K$-theory computations of \cite{yang1998crossed} and easily computable abelianizations of each group to prove that all $17$ wallpaper groups are $C^*$-superrigid.

We denote the center of a $C^*$-algebra $A$ by $\ZZ(A)$. 

For any group $G$, we write $C_x$ to denote the conjugacy class of $x\in G$. An \emph{$FC$-group} is a group where every element has finite conjugacy classes (i.e., $|C_x|<\infty$ for all $x\in G$). Finite groups, as well as abelian groups, are examples of $FC$-groups. For an arbitrary group $G$, its \emph{$FC$-center} (denoted $FC(G)$), is the set of all elements of $G$ that have finite conjugacy class. We note that $FC(G)$ is a characteristic (hence normal) subgroup of $G$ (\cite[14.5.5]{robinson_group_theory_book}) and is, by definition, an $FC$-group in its own right.

In the case that $G$ is an $FC$-group, the elements of $G$ of finite order form a characteristic subgroup which will call $Tor(G)$. We observe that $Tor(G)$ is locally finite (\cite[14.5.7]{robinson_group_theory_book}). Moreover, a torsion free $FC$-group is automatically abelian (\cite[Prop. 2.5]{tf_2_step_c*_superrigid}).

A key property which characterizes the $C^*$-algebra $C^*(G)$ of a crystallographic group is that its center contains only trivial projections. We start by connecting this $C^*$-algebraic property with a group property.

\begin{prop}\label{characterization of projectionless center}
    Let $G$ be a discrete group. The following are equivalent:
    \begin{enumerate}[label=(\roman*).]
        \item $P(\ZZ(C^*(G)))=\{0,1\}$
        \item $FC(G)$ is torsion-free.
    \end{enumerate}
    \begin{proof} {\em ((i)$\Rightarrow$(ii))}:  By way of contraposition, suppose that $FC(G)$ has torsion. Set $T=Tor(FC(G))$ where, by assumption, $T\neq \{1_T\}$. Let $x\in T\setminus\{1_T\}$ and consider $y:=\sum_{g\in C_x}g\in C^*(G)$. Notice that $y\in \ZZ(C^*(G))$. Indeed, a direct computation shows that $y$ commutes with all $g\in G$ and so $y$ is in the center of $C^*(G)$ by the fact that $C^*(G)$ is generated (as a $C^*$-algebra) from the elements of $G$. Moreover, because $C_x$ is a finite set, $y\in C^*(T_0)$ where $T_0$ is a finitely generated subgroup of $T$. But $T$ is locally finite, hence $T_0$ must be finite. It follows that $A:=\ZZ(C^*(G))\cap C^*(T_0)$ is a finite dimensional $C^*$-subalgebra of $\ZZ(C^*(G))$ and $A\neq \C$ (because $y\in A$). Hence $\ZZ(C^*(G))$ must have non-trivial projections. \par

{\em ((ii)$\Rightarrow$(i))}: Assume that $FC(G)$ is torsion free. Because $FC(G)$ is an $FC$-group, it must be abelian. Hence, its Pontryagin dual, $\widehat{FC(G)}$, is connected. This implies that $C^*(FC(G))\cong C(\widehat{FC(G)})$ has no non-trivial projections. Moreover, $\ZZ(C^*(G))\subseteq C^*(FC(G))$ by the proof of \cite[Prop. 2.6]{tf_2_step_c*_superrigid}. It follows that $P(\ZZ(C^*(G)))=\{0,1\}$.
    \end{proof}
\end{prop}
The following Lemma is well-known to experts, but we present a proof for the sake of completion.
\begin{lemma}\label{fixed point subgroup of non-trivial autom}
 Let $N$ be a torsion free abelian group and $\sigma\in Aut(N)$ be an automorphism. Assume that the fixed point subgroup of $\sigma$ has finite index in $N$. Then $\sigma=id$.
 \begin{proof}
     Notice that $\phi:N\to N$ defined via $\phi(x)=x(\sigma(x))^{-1}$ is a well-defined group homomorphism with finite image. The latter is true because $\phi$ factors through the fixed point subgroup of $\sigma$.
     Since $\phi(N)$ is a subgroup of the torsion free group $N$, it has to be torsion free. It follows that $\phi(N)=\{1_{N}\}$. Hence, for every $x \in N$ we have $x(\sigma(x))^{-1}=\{1_N\} \Rightarrow \sigma(x)=x$.

 \end{proof}

\end{lemma}
We now characterize the virtually abelian groups that have torsion-free $FC$-center. 
\begin{prop}\label{characterization of tf fc center}
    Let $G$ be a virtually abelian group. The following are equivalent:
    \begin{enumerate}[label=(\roman*).]
        \item $FC(G)$ is torsion free.
        \item $G$ has a torsion free, normal subgroup $N$ that is maximally abelian in $G$, and $[G:N]<\infty.$
        \item $FC(G)$ is a torsion free, maximally abelian in $G$, and $[G:FC(G)]<\infty$.
    \end{enumerate}
    \begin{proof}
        {\em ((i)$\Rightarrow$ (ii))}: Assume that $FC(G)$ is torsion free. Because $G$ is virtually abelian, there exists $N\unlhd G$ where $N$ is abelian and $[G:N]<\infty.$ Because $N\subseteq FC(G)$, it follows that $N$ must be torsion free. Set $L:=C_G(N)$. If $L=N$, then $L$ is maximally abelian in $G$, hence we are done. Assume now that $L\gneq N$. Notice that $L\leq FC(G)$, hence $L$ is torsion free. Thus $L$ is a torsion free, $FC$ group, so it has to be abelian. Notice that $[G:L]<\infty$ and that $L$ is maximally abelian in $G$ by construction. \par
        {\em ((ii)$\Rightarrow$(iii))}: Let $N$ be as in the assumption. It is enough to show that $FC(G)=N$. Because the action via conjugation $G\curvearrowright N$ factors through the finite group $G/N$ (see also Section 2.4), we deduce that  $N\leq FC(G)$. Assume for the sake of contradiction that there exists $x\notin N$ with finite conjugacy class. Because $N$ is maximally abelian in $G$, the automorphism $\sigma_x:N\to N$ defined via $\sigma_x(n)=x^{-1}nx$ is not trivial. Moreover, its fixed point set has finite index in $N$. However, this cannot happen because of Lemma \ref{fixed point subgroup of non-trivial autom}, and the fact that $N$ is torsion free abelian. \par
        {\em ((iii)$\Rightarrow$(i))}: This is immediate.
    \end{proof}
\end{prop}
We also need to explicitly compute the center of the $C^*$-algebra of certain virtually abelian groups. We note that this result will appear on the in-preparation paper \cite{Knubyetal}. We would like to thank Jakub Curda for pointing out this result to us.
\begin{prop}\label{computing the center of certain va groups}
    Let $G$ be a virtually abelian group, and $N\unlhd G$ such that $N$ is torsion free, maximally abelian and $D:=G/N$ is a finite group. Then $\ZZ(C^*(G))=C(\widehat{N}/D)$.
    \begin{proof}
        Proposition \ref{characterization of tf fc center}{\em ((ii)$\Rightarrow$(iii))} implies that $FC(G)=N$. Because $N$ is torsion free, by the proof of \cite[Prop. 2.6]{tf_2_step_c*_superrigid}, it follows that 
        \[\ZZ(C^*(G))=C^*(N)^D=\left\{\sum_{h\in N}\lambda_h \delta_h\,\colon\, \lambda_h\text{ is constant on conjugacy classes}\right\}\]
        Let $\mathcal{F}:C^*(N)\rightarrow C(\widehat{N})$ be the Fourier transform. For $x=\sum_{h\in N}\lambda_h \delta_h$, $\chi \in \widehat{N}$ and $g\in D$, we have that
        \[\mathcal{F}(x)(g\cdot \chi)=\sum_{h\in N}\lambda_h(g\cdot \chi)(h)=\sum_{h\in N}\lambda_h\chi(\gamma(g)h\gamma(g)^{-1})=\sum_{h\in N}\lambda_{\gamma(g)^{-1}h\gamma(g)}\chi(h).\]
        It follows that $\lambda_h$ is constant on conjugacy classes if and only if $x$ satisfies the fact that $\mathcal{F}(x)(\chi)=\mathcal{F}(x)(g\cdot \chi)$ for every $\chi\in \widehat{N}$ and every $g\in D$. It follows that, after identifying $C^*(N)$ with $C(\widehat{N})$ via the Fourier transform, 
        \[C^*(N)^D=\left\{f\in C(\widehat{N}) \,\colon\, f(\chi)=f(g\cdot \chi)\text{ for every }g\in D \text{ and }\chi\in \widehat{N}\right\}.\] Observe that the right hand side on the above equality is isomorphic to $C^*(\widehat{N}/D)$ so proof is complete.
    \end{proof}
\end{prop}
  Let $G$ be a finitely generated, virtually abelian group and $N\unlhd G$ with $N\cong \Z^r$ and $[G:N]<\infty$. \par
  Set $$S:=\{\chi \in \widehat{N}: \, |\mc{O}_{\chi}|=|G/N|\}\subseteq \widehat{N}$$ Observe that $D:=G/N$ acts freely, and by homeomorphisms\footnote{The explanation for this is identical to the one that is presented below Lemma \ref{desired orbit is open} for a different action.} on $S$. 
   Moreover, $S$ is open and dense in $\widehat{N}.$ Indeed, open follows almost identically as Lemma \ref{desired orbit is open} (see also \cite[Prop. 4.12]{CW24}). Density follows almost identically as Proposition \ref{density} (see also \cite[Lemma 2.1]{strongly_qd_eckhardt}). We show that the quotient map $q:S\to S/D$ is a local homeomorphism. This is well-known to experts, but we present a proof of the following more general result for the sake of completion.

\begin{prop}\label{actions with finite groups give local isom}
Let $D$ be a finite group acting freely and by homeomorphisms on a locally compact, metrizable space $X$. Then the quotient map $q:X\rightarrow X/D$ is a local homeomorphism.
\begin{proof}
    Let $d$ be a metric making $X$ metrizable. We already know that $q$ is continuous, open, closed and surjective. So, it is enough to show that for every $x\in X$, there exists an open neighborhood $U_x$ such that the restriction of $q$ on $U_x$ is injective. However, injectivity of $q$ on $U_x$ is equivalent to saying that no two elements of $U_x$ are on the same orbit. Fix $x\in X$ and let 
    \[m:=\min\{d(x,y)\colon y \text{ is on the orbit of }x \}.\]
    Because $|D|<\infty$ and the action is free, $m$ is well-defined and $m>0$. \par
    For every $g\in D$, consider $V_{g,x}:=B(gx, \frac{m}{4})$. Because $D\curvearrowright X$ is an action by homeomorphisms, $g^{-1}V_{g,x}$ is open. Hence, there is a $\delta_g>0$, such that $B(x,\delta_g)\subset g^{-1}V_{g,x}$. Set $$\delta:=\min\l\{\delta_g: g\in D, \frac{m}{4} \r\}$$
    We will show the desired property for $U_x:=B(x,\delta)$. Let $y\in B(x,\delta)$ and it notice that it is enough to show that $d(y,gy)>2\delta$ for every $g\in D$. For every $g\in D$ we have
    $$d(y,gy)\geq d(x,gx)-d(x,y)-d(gx,gy)$$
    
    $$d(y,gy)>m-\frac{m}{2}-\frac{m}{2}=\frac{m}{2}\geq 2\delta$$
    \begin{itemize}
        \item For the first inequality we used the reverse triangle inequality.
        \item For the second inequality we use that $y\in B(x,\delta)$, so $d(x,y)<\delta\leq \frac{m}{4}$. We also used the fact that $gB(x,\delta)\subset gB(x,\delta_g)\subset gg^{-1}V_{g,x}=V_{g,x}$. So $d(gx,gy)<\frac{m}{4}$.
        \item For the last inequality we used that $\delta\leq \frac{m}{4}$ by construction.
    \end{itemize}
    
\end{proof}
\end{prop}
We also need the following Remark.
\begin{rmk}\label{locally_euclidean_homogeneous}
    Let $G$ be a topological group and assume that there exists $g_0\in G$ that has a Euclidean neighborhood, say $V_0$. Notice that for every $g\in G$, $V_g=gg_0^{-1}V_0$ is a Euclidean neighborhood of $g$. It follows that $G$ is locally Euclidean.
\end{rmk}
 We are now ready to prove the first of our main results in this section.

 It has to be noted that some of the arguments used in the proof below are also used in a Theorem of the in-preparation paper \cite{Knubyetal}.
\begin{thm}\label{crystallographic_is_preserved}
    Let $G$ be a crystallographic group and $H$ a discrete group such that $C^*(G)\cong C^*(H)$. Then 
    \begin{enumerate}[label=(\roman*.)]
        \item $H$ is crystallographic.
        \item $h(G)=h(H)$.
        \item If $D_1$ and $D_2$ are the point groups of $G$ and $H$, respectively, then $|D_1|=|D_2|$.
    \end{enumerate}
    
    \begin{proof}\mbox{}
    
        {\em (i.)} First of all, notice that $H$ is virtually abelian by \cite{Tho64}. Proposition \ref{characterization of tf fc center} implies that $FC(G)$ is torsion free, and thus $P(\ZZ(C^*(H)))=P(\ZZ(C^*(G)))=\{0,1\}$ by Proposition \ref{characterization of projectionless center}. Again, Propositions \ref{characterization of projectionless center} and \ref{characterization of tf fc center}, applied to $H$ this time, imply that there exists a torsion free, normal $L$ such that it is a maximally abelian subgroup of $H$, and $D_2:=H/L$ is finite. \par
        Because $G$ is crystallographic, it has a normal, maximally abelian subgroup $N\cong \Z^r$ such that $D_1:=G/N$ is finite. By applying Proposition \ref{computing the center of certain va groups} to both $G$ and $H$, we deduce that
        $$C(\widehat{N}/D_1)\cong \ZZ(C^*(G))\cong \ZZ(C^*(H))\cong C(\widehat{L}/D_2)$$
        Because $\widehat{N}/D_1$ and $\widehat{L}/D_2$ are both compact and Hausdorff, it follows that 
        \begin{equation}\label{1}
           \widehat{N}/D_1\cong \widehat{L}/D_2. 
        \end{equation}
        
        Notice that $\widehat{N}\cong \TT^r$ is locally Euclidean. Hence, Proposition \ref{actions with finite groups give local isom}, the comments above it, and (\ref{1}) imply that $\widehat{L}/D_2$ contains an open dense subset that it is locally Euclidean. Again, Proposition \ref{actions with finite groups give local isom}, and the comments above it imply that $\widehat{L}$ has an open dense subset that is locally Euclidean. Now, the fact that $\widehat{L}$ is locally Euclidean follows from Remark \ref{locally_euclidean_homogeneous}. Recall that $\widehat{L}$ is also a connected (because $L$ is torsion-free), compact, abelian group. It follows from the structure theory for locally compact abelian groups \cite[Thm. 4.2.4]{principles_of_harm_anal_book} that $\widehat{L}\cong \TT^s$. Thus $L\cong \Z^s$, which implies that $H$ is crystallographic. 
        
        {\em (ii.)} This follows from Theorem \ref{main_result_2}, However, we can also prove it as follows:
\[h(G)=r=\dim(\widehat{N}/D_1)=\dim(\widehat{L}/D_2)=\dim(\widehat{L})=s=h(H).\]

        {\em (iii.)} This follows from Corollary \ref{crystal_like_order_of_point_group} and the fact that crystallographic groups form crystal-like lattice pairs.
    \end{proof}
\end{thm}
We will now use the above theorem to produce new examples of $C^*$-superrigid groups.  \par
Recall from the introduction and Example \ref{example:crystallographic_group_defn} that the wallpaper groups are defined to be the crystallographic groups with Hirsch length two, and that there are 17 of them.

\begin{thm}\label{wallpaper_c*_superrigid}
    All 17 crystallographic groups are $C^*$-superrigid.
\end{thm}

\begin{proof}
    Let $G$ be a wallpaper group and $H$ be a discrete group such that $C^*(G)\cong C^*(H)$. Theorem \ref{crystallographic_is_preserved} ensures that $H$ is a wallpaper group. This narrows the possible isomorphism classes for $H$ down to 17. Furthermore, the following must also be true.
\begin{itemize}
    \item The point groups of $G$ and $H$ have the same order (Theorem \ref{crystallographic_is_preserved}).
    \item $K_i(C^*(G))\cong K_i(C^*(H))$
    \item $C^*(G_{ab})\cong C^*(H_{ab})$ (\cite[Cor. 1.3]{free_nilp_c*_superrigid}). In particular, $|Tor(G_{ab})|=|Tor(H_{ab})|$.
\end{itemize}
The K-theory of the group $C^*$-algebras of all wallpaper groups has been computed in \cite{yang1998crossed}. Moreover, their point groups and abelianizations are well-known. For example, they can be computed using the computer algebra system GAP \cite{GAP4}. To complete the proof we compare the data in Table \ref{tab: wall} observing that for any distinct pair of wallpaper groups one of the above bullet points is not satisfied. We conclude that $G\cong H$.
\end{proof}

%\newpage
   
\begin{table}[h!]
\begin{tabular}{|C|C||C|C|C|C|C|C}\hline
\text{\#} & G & K_0(G) & K_1(G) & H_1(G)\cong G_{ab} & \te{Point}{\te{ Group}}\\\hline\hline
1^* & p1 &\Z^2 &\Z^2 & \zn^2&\{e\}\\\hline
2 & p2 &\Z^6 &0 & \zn_2\times\zn_2\times \zn_2 &\Z_2\\\hline
3 & pm &\Z^3 &\Z^3 & \zn\times\zn_2\times\zn_2&\Z_2\\\hline
4^* & pg &\Z &\Z\times \Z_2 & \zn\times\zn_2 & \Z_2\\\hline
5 & cm & \Z^2& \Z^2& \zn\times\zn_2 &\Z_2\\\hline
6 & p2mm & \Z^9& 0& \zn_2\times\zn_2\times\zn_2\times\zn_2 & \Z_2\times \Z_2 \\\hline
7 & p2mg & \Z^4& \Z&\zn_2\times\zn_2\times\zn_2& \Z_2\times \Z_2\\\hline
8 & p2gg &\Z^3 &\Z_2 &\zn_4\times\zn_2 & \Z_2\times \Z_2\\\hline
9 & c2mm &\Z^5 &0 & \zn_2\times\zn_2\times\zn_2&\Z_2\times \Z_2\\\hline
10 & p4 & \Z^9& 0& \zn_4\times \zn_2& \Z_4\\\hline
11 & p4mm & \Z^9&0 & \zn_2\times\zn_2\times \zn_2& D_{2\cdot4} \\\hline
12 & p4gm &\Z^6 &0 & \zn_4\times\zn_2  & D_{2\cdot4} \\\hline
13 & p3 &\Z^8 & 0& \zn_3\times\zn_3& \Z_3\\\hline
14 & p3m1 &\Z^5 &\Z & \zn_2& S_3 \\\hline
15 & p31m &\Z^5 & \Z& \zn_3\times \zn_2& S_3\\\hline
16 & p6 & \Z^{10}&0 & \zn_3\times \zn_2& \Z_6\\\hline
17 & p6mm &\Z^8 & 0&\zn_2\times \zn_2 & D_{2\cdot6}\\\hline
\end{tabular}
\vspace{.1in}
    \caption{$(*)$ indicates the group is torsion-free, i.e., a Bieberbach group.}
    \label{tab: wall}
\end{table}

\newpage

\printbibliography

\end{document}